\documentclass[oneside,english]{amsart}
\usepackage[T1]{fontenc}
\usepackage[latin9]{inputenc}
\usepackage{amsthm}
\usepackage{amstext}
\usepackage{amssymb}
\usepackage{esint}
\usepackage{amscd}
\usepackage{mathrsfs}
\usepackage{color}
\makeatletter
\numberwithin{equation}{section}
\numberwithin{figure}{section}
\theoremstyle{plain}
\newtheorem{thm}{\protect\theoremname}[section]

\theoremstyle{plain}
\theoremstyle{definition}

\theoremstyle{plain}
\newtheorem{lem}[thm]{\protect\lemmaname}
\newtheorem{cor}[thm]{\protect\corollaryname}
\theoremstyle{plain}

\theoremstyle{plain}
\makeatother

\usepackage{babel}
\providecommand{\definitionname}{Definition}
\providecommand{\lemmaname}{Lemma}
\providecommand{\theoremname}{Theorem}
\providecommand{\corollaryname}{Corollary}
\providecommand{\remarkname}{Remark}
\providecommand{\propositionname}{Proposition}
	
\DeclareMathOperator{\loc}{loc}
\DeclareMathOperator{\supp}{supp}

\DeclareMathOperator{\diam}{diam}

\DeclareMathOperator{\ess}{ess}

\DeclareMathOperator{\cp}{cap}
\DeclareMathOperator{\Cp}{Cap}

\DeclareMathOperator{\Lip}{Lip}
\DeclareMathOperator{\md}{mod}

\begin{document}

\title[On mappings generating embedding operators in Sobolev classes]{On mappings generating embedding operators in Sobolev classes on metric measure spaces}

\author{Alexander Menovschikov and Alexander Ukhlov}
\begin{abstract}
In this article, we study homeomorphisms $\varphi: \Omega \to \widetilde{\Omega}$ that generate embedding operators in Sobolev classes on metric measure spaces $X$ by the composition rule $\varphi^{\ast}(f)=f\circ\varphi$. In turn, this leads to Sobolev type embedding theorems for a wide class of bounded domains $\widetilde{\Omega}\subset X$.
\end{abstract}
\maketitle

\section{Introduction}

Let $\Omega, \widetilde{\Omega} \subset \mathbb{R}^n$, $n \geq 2$, be Euclidean domains. Then a homeomorphism $\varphi: \Omega \to \widetilde{\Omega}$ generates a bounded embedding operator on seminormed Sobolev spaces
$$
\varphi^{\ast}: L^1_p(\widetilde{\Omega}) \to L^1_q(\Omega), \quad 1 < q \leq p < \infty,
$$
by the composition rule $\varphi^{\ast}(f) = f \circ \varphi$, if there exists a constant $C_{p,q}$ such that the inequality
$$
\|\varphi^{\ast}(f)\|_{L^1_q(\Omega)} \leq C_{p,q} \|f\|_{L^1_p(\widetilde{\Omega})}
$$
holds for any function $f \in L^1_p(\widetilde{\Omega})$ \cite{VU04, VU05}. In accordance with the non-linear potential theory \cite{HM72}, we consider functions of Sobolev spaces defined up to a set of $p$-capacity zero.

The embedding operators on Sobolev spaces, generated by the composition rule $\varphi^\ast(f) = f \circ \varphi$, trace back to the works \cite{M69, Sl41, VGR79}. These operators have a significant role in the theory of Sobolev embeddings \cite{GGu, GU} and in the spectral theory of elliptic operators \cite{GPU18, GU17}. In \cite{U93, VU02} (see, also \cite{V88} for the case $p=q$), necessary and sufficient conditions were obtained for mappings of Euclidean domains $\varphi: \Omega \to \widetilde{\Omega}$ that generate bounded composition operators on Sobolev spaces:
$$
\varphi^{\ast}: L^1_p(\widetilde{\Omega}) \to L^1_q(\Omega), \quad 1 < q \leq p < \infty.
$$

The foundations of the geometric theory of composition operators on Sobolev spaces were established in \cite{VU04, VU05}. Over the last decade, this theory has continued to develop intensively, as seen in \cite{MU24_1, MU24_2, T15, V20,VP24}.

The composition operators on Sobolev spaces defined on stratified Lie groups \cite{FS} were considered in \cite{VU98, VU04, VU05} within the framework of the quasiconformal mappings theory on spaces with non-Riemannian metrics \cite{KR95, Pa}.

The Sobolev spaces associated with non-Riemannian metrics arise in the study of elliptic boundary value problems \cite{Ho67} and have been intensively studied in recent decades, see for example, \cite{BLU,N01}. To study the pointwise behavior of Sobolev functions, a non-coordinate approach to the Sobolev spaces theory was suggested in \cite{H93, V96}. In subsequent works \cite{HK00, HKST, Sh00}, the theory of Sobolev spaces on metric measure spaces was constructed. The variational problems on metric measure spaces see, for example, in \cite{GK}.

Let $X = (X, \rho, \mu)$ be a doubling metric measure space which supports the \textit{weak $p$-Poincar\'e inequality} \cite{HK00}.  In this article, using methods of non-linear potential theory \cite{HM72}, we consider bi-measurable homeomorphisms $\varphi: \Omega \to \widetilde\Omega$, of bounded domains $\Omega,\Omega\subset X$, which generate bounded composition operators on Newtonian--Sobolev spaces
$$
N^{1,p}(\widetilde{\Omega}) \to N^{1,q}(\Omega), \quad 1 < q \leq p < \infty.
$$

We prove the Luzin $N^{-1}$-property of such mappings with respect to capacities and obtain necessary and sufficient conditions on bi-measurable homeomorphisms generating bounded composition operators on Newtonian--Sobolev spaces.

The main problem that we solve is the characterization of bi-measurable homeomorphisms, which generate bounded composition operators on Newtonian--Sobolev spaces, into the corresponding Reshetnyak-Sobolev classes. We prove this by using special test functions generated by distance functions. In the recent work \cite{Ro23}, bi-composition operators on Newtonian--Sobolev spaces were considered. However, in that work, the author used the definition of Sobolev classes of mappings suggested in \cite{GU10, U04} for mappings in Euclidean space. The equivalence of this definition to other definitions of Sobolev classes in general metric structures, such as Carnot groups, remains an open problem.

On the base of the composition operators we consider Sobolev type embedding theorems in weak $(p,q)$-quasiconformal $\alpha$-regular domains $\widetilde{\Omega}\subset X$, $1<q\leq p<\infty$, $\alpha > 1$, which are defined in Section 4.3. The first result states:

\vskip 0.1cm
\noindent
\textit{Let $\widetilde{\Omega}\subset X = (X, \rho, \mu)$ be a weak $(p,q)$-quasiconformal $\alpha$-regular domain, $1<q<\nu$. Then the embedding operator
 $$
    \widetilde{i}: N^{1,p}(\widetilde{\Omega}) \to L^{s}(\widetilde\Omega)
 $$
 is bounded for $s\leq r\frac{\alpha-1}{\alpha}$, $r=\frac{\nu q}{\nu-q}$, and is compact for $s< r\frac{\alpha-1}{\alpha}$, where $\nu > 0$ is  the order of a relative lower volume decay of $X$
}
\vskip 0.1cm

In particular, let us formulate the embedding theorem in the case of Lipschitz domains.

\vskip 0.1cm
\noindent
\textit{Let $\widetilde{\Omega}\subset X = (X, \rho, \mu)$ be a Lipschitz domain. Then for $1<p<\nu$, the embedding operator
 $$
    \widetilde{i}: N^{1,p}(\widetilde{\Omega}) \to L^{s}(\widetilde\Omega)
 $$
 is bounded for $s\leq r$, $r=\frac{\nu p}{\nu-p}$, and is compact for $s< r$.}
\vskip 0.1cm

In metric measure spaces we call a domain $\widetilde{\Omega}\subset X = (X, \rho, \mu)$ as a Lipschitz domain, if there exists a bi-Lipschitz homeomorphism $\varphi: \widetilde{B} \to \widetilde{\Omega}$, where $\widetilde{B}$ is the unit ball in $X$.

\section{Sobolev classes on metric measure spaces}

\subsection{Sobolev functions on metric measure spaces}

Given a metric space $X=(X, \rho)$, we consider a \textit{doubling measure} $\mu$, i.e. a Borel regular measure, such that each ball in $X$ has a positive and finite measure and the inequality
$$
 \mu(B(x,2r)) \leq C_{\mu} \mu(B(x,r))
$$
holds for all $x \in X$ and $r>0$ with a (smallest) constant $C_{\mu} \geq 1$. Then $X = (X, \rho, \mu)$ is called a \textit{doubling metric measure space}.

A metric measure space $X = (X, \rho, \mu)$ is said to satisfy a \textit{relative lower volume decay} of order $\nu > 0$ if there is a constant $C \geq 1$ such that
$$
	\frac{\mu(B)}{\mu(B_0)} \geq C \Big(\frac{r}{r_0}\Big)^\nu
$$
whenever $B_0$ is an arbitrary ball of radius $r_0$ and $B=B(x,r)$, $x \in B_0$, $r \leq r_0$. It was shown in \cite[Lemma 8.1.13]{HKST}, that doubling metric measure spaces support this property with
$$
    \nu = \log_2 C_\mu.
$$

Let $\Gamma$ be a family of curves in $(X, \rho, \mu)$. Denote by $adm(\Gamma)$ the set of Borel functions (\textit{admissible functions}) $\eta: X \to [0,\infty]$ such that the inequality
$$
\int\limits_{\gamma}\eta \, ds \geq 1
$$
holds for locally rectifiable curves $\gamma\in\Gamma$.

The quantity
$$
\md_p(\Gamma) = \inf\int\limits_{X} \eta^{p} \, d\mu
$$
is called the \textit{$p$-module of the family of curves} $\Gamma$ \cite{MRSY09}. The infimum is taken over all admissible functions $\eta \in adm(\Gamma)$.

Recall the notion of Sobolev spaces on doubling metric measure spaces \cite{HK00,Sh00}. The \textit{Newtonian-Sobolev space} $N^{1,p}(X)$, $1<p<\infty$, is defined as a space of measurable functions $f : X \to \mathbb{R}$ such that $f \in L^p(X)$ and there exists a Borel function $g: X \to [0, \infty]$ such that $g \in L^p(X)$ and 
\begin{equation}\label{Newtoneq}
|f(\gamma(a))-f(\gamma(b))| \leq \int_\gamma g \, ds
\end{equation}
for $p$-a.e. rectifiable curve $\gamma: [a,b] \to X$, where the integral on $\gamma$ denotes the line integral of $g$ along $\gamma$ and $p$-a.e. means that the property holds for all curves except a family $\Gamma$ with $\md_p \Gamma = 0$ \cite{MRSY09}.

The function $g$ which satisfies \eqref{Newtoneq} for $p$-a.e. rectifiable curve $\gamma: [a,b] \to X$ is called a \textit{$p$-weak upper gradient of $f$} and is denoted as $\nabla_{up} f$. If a weak upper gradient $g$ satisfies \eqref{Newtoneq} for every rectifiable curve $\gamma: [a,b] \to X$, then it is called an \textit{upper gradient of $f$}. By Mazur's lemma and by Fuglede's lemma \cite{Sh00} there exists a unique \textit{minimal $p$-weak upper gradient $\nabla_{min} f \in L^p(X)$}, $1<p<\infty$, defined up to a set of measure zero.

The Newtonian-Sobolev space $N^{1,p}(X)$ is a Banach space with the norm
$$
\|f\|_{N^{1,p}(X)} = \|f\|_{L^p(X)} + \inf\limits_{\nabla_{up} f}\|\nabla_{up} f\|_{L^p(X)},
$$
where the infimum is taken over all $p$-weak upper gradients (equivalently, over all upper gradients). This norm is equivalent to the norm
$$
\|f\|_{N^{1,p}(X)} = \|f\|_{L^p(X)} + \|\nabla_{min} f\|_{L^p(X)}, \quad 1<p<\infty.
$$

Let $X$ be a doubling metric measure space. Then every function $f$ that belongs to $N^{1,p}(X)$ has a representative that is well-defined up to a set of a Sobolev $p$-capacity zero \cite{GT02,KM96}. 

Recall the notion of the $p$-capacity of a set $E\subset X$ \cite{BB,GT02}. Suppose $F\subset X$ is a compact set. The \textit{$p$-capacity of $F$} is defined as
\begin{equation*}
 \Cp_p(F;X) = \inf\left\{ \int_X |u|^p\,d\mu + \int_X (\nabla_{min} u)^p\,d\mu \right\},
\end{equation*} 
where the infimum is taken over all functions $u\in C_0(X)\cap N^{1,p}(X)$ such that $u\geq 1$ on $F$. These functions $u$ are called \textit{admissible functions} for the compact set $F\subset X$. 
If 
$U\subset X$ 
is an open set, we define
\begin{equation*}
\Cp_{p}(U;\Omega)=\sup_F\{\Cp_{p}
(F;X)\,:\,F\subset U,\,\, F\,\,\text{is compact}\}.
\end{equation*}

In the case of an arbitrary set 
$E\subset X$
we define the inner $p$-capacity 
\begin{equation*}
\underline{\Cp}_{p}(E;X)=\sup_F\{\Cp_{p}(F;X)\, :\,\,F\subset E\subset X,\,\, F\,\,\text{is compact}\},
\end{equation*}
and the outer $p$-capacity 
\begin{equation*}
\overline{\Cp}_{p}(E;X)=\inf_U\{\Cp_{p}(U;X)\, :\,\,E\subset U\subset X,\,\, U\,\,\text{is open}\}.
\end{equation*}

Then a set $E\subset X$ is called $p$-capacity measurable, if 
$$
\underline{\Cp}_p(E;X)=\overline{\Cp}_p(E;X).
$$
Let $E\subset X$ be a $p$-capacity measurable set. The value
$$
\Cp_p(E; X)=\underline{\Cp}_p(E; X)=\overline{\Cp}_p(E;X)
$$
is called the $p$-capacity measure of the set $E\subset X$. By \cite{Ch54}, Borel sets are $p$-capacity measurable sets.

One of the important property of metric measure spaces is supporting of Poincar\'e inequality. The metric measure space $X = (X, \rho, \mu)$ supports the \textit{weak $p$-Poincar\'e inequality} \cite{HK00}, if for all balls $B \subset X$, for all functions $f \in L^1(\sigma B)$, $\sigma>1$, and for all $p$-weak upper gradients $\nabla_{up} f$ of $f$ in the ball $\sigma B$ the following inequality holds:
    \begin{equation}\label{Poincareineq}
        \frac{1}{\mu(B)} \int_B |f - f_B| \, d\mu \leq C_p (\diam(B))\left(\frac{1}{\mu(\sigma B)} \int_{\sigma B} (\nabla_{up} f)^p \, d\mu\right)^{\frac{1}{p}}
    \end{equation}
with the constant $C_p$ that does not depend on $B$, $f$, and $\nabla_{up} f$. 

If $X$ is a doubling metric measure space that supports the weak $p$-Poincar\'e inequality, then for every function $f\in N^{1,p}(X)$ there exists a quasicontinuous representative $\bar{f}$ of this function (see \cite{HKST}) and Lipschitz functions are dense in $N^{1,p}$ \cite{H96,Sh00}.

Let $\Omega \subset X$ be a bounded domain. We consider $\Omega$ as a doubling metric measure space with the metric $\rho$ and measure $\mu$ induced from $X$. The \textit{seminormed Newtonian-Sobolev space} $S^{1,p}(\Omega)$, for $1 < p < \infty$, is defined as a space of measurable functions $f : \Omega \to \mathbb{R}$ such that $f \in L^1_{\loc}(\Omega)$, and there exists a minimal $p$-weak upper gradient $\nabla_{min} f \in L^p(\Omega)$, with the finite seminorm:
    $$
        \|f\|_{S^{1,p}(\Omega)} = \|\nabla_{min} f\|_{L^p(\Omega)}.
    $$

It has been proved \cite[Theorem 9.1.15]{HKST} that a doubling metric measure space $X = (X, \rho, \mu)$, which supports a weak $p$-Poincar\'e inequality, also supports the standard Sobolev embeddings on balls. Consequently,
$$
N^{1,p}(B) = S^{1,p}(B), \,\,1<p<\infty,
$$
for any ball $B$ in $X = (X, \rho, \mu)$.

 Moreover, we can consider the $p$-capacity based on the seminorm
    $$
        \cp_p(F;\Omega) = \inf \int_{\Omega} (\nabla_{min} u)^p\,d\mu = \inf \|u\|^p_{S^{1,p}(\Omega)},
    $$
where the infimum is taken over all admissible functions $u \in N^{1,p}(\Omega)$ such that $u \geq 1$ on the compact set $F$. According to \cite{VCh1,VCh2}, this capacity is also a Choquet capacity \cite{Ch54} and in doubling metric measure space $X$ this two capacities have the same null-sets, i.e. 
$$
\Cp_p(E;\Omega)=0\,\,\text{if and only if}\,\, \cp_p(E;\Omega)=0.
$$
In it turns this implies that in $S^{1,p}(\Omega)$ every function also has a quasicontinuous representative and Lipschitz functions are dense in $S^{1,p}(\Omega)$.

\subsection{Sobolev mappings on metric measure spaces}

Let us recall the definition of Sobolev mappings between metric measure spaces \cite{R97}. Let $X=(X, \rho, \mu)$ and $\widetilde{X}=(\widetilde{X}, \widetilde{\rho}, \widetilde{\mu})$ be complete doubling metric measure spaces, $\mu(X)<\infty$. We consider bi-measurable mappings in the sense of \cite{Z69} 
$$
\varphi: X \to \widetilde{X},
$$
where a bi-measurability means that the mapping $\varphi$ has the Luzin $N$-property and the Luzin $N^{-1}$-property. 
Recall that a mapping $\varphi: X \to \widetilde{X}$ has the Luzin $N$-property if an image of a set of $\mu$-measure zero has $\widetilde{\mu}$-measure zero and has the Luzin $N^{-1}$-property if a pre-image of a set of $\widetilde{\mu}$-measure zero has $\mu$-measure zero.

The class $L^p(X; \widetilde{X})$, $1 < p < \infty$, denotes bi-measurable mappings $\varphi: X \to \widetilde{X}$, such that functions $[\varphi]_z := \widetilde{\rho}(\varphi(x), z) \in L^p(X)$ for a fixed point $z \in \widetilde{X}$. This definition is independent of the choice of the point $z\in \widetilde{X}$.

There are several equivalent definitions of the Sobolev mappings with values in metric spaces. The first approach is based on compositions with Lipschitz functions \cite{R97,R06,VU98}. 
A function $f: \widetilde{X}\to\mathbb R$ belongs to the Lipschitz space $\Lip(\widetilde{X})$ if 
    $$
        \|f\|_{\Lip(\widetilde{X})}=\sup\limits_{x,y\in \widetilde{X}}\frac{|f(x)-f(y)|}{\widetilde{\rho}(x,y)}<\infty.
    $$

Now let $\Omega$ and $\widetilde{\Omega}$ be bounded domains in the Euclidean space $\mathbb R^n$. If a bi-measurable mapping $\varphi:\Omega\to\widetilde{\Omega}$ belongs to the Sobolev class $W^1_p(\Omega;\widetilde{\Omega})$, then the chain rule
\begin{equation*}
\frac{\partial \left(f\circ\varphi\right)}{\partial x_i}(x)=\sum\limits_{j=1}^n \frac{\partial f}{\partial y_j}(\varphi(x)) \frac{\partial \varphi_j}{\partial x_i}(x)\,\,\text{a.e. in}\,\,\Omega,\,\,i=1,...,n,
\end{equation*}
holds for every function $f\in \Lip(\widetilde{\Omega})$ and weak derivatives $\frac{\partial \left(f\circ\varphi\right)}{\partial x_i}$, $i=1,...,n$, belong to the Lebesgue space $L^p(\Omega)$  \cite{AM,VGR79}. Placing this property in the base of the definition of Sobolev metric-valued mappings, we come to the following definition \cite{R97,R06,VU98}.

The \textit{Reshetnyak-Sobolev class} $V^{1,p}(X;\widetilde{X})$, $1 < p < \infty$, is a class of bi-measurable mappings $\varphi\in L^p(X; \widetilde{X})$ such that:

\noindent
$(1)$ for every point  $z \in \widetilde{X}$ the function $[\varphi]_{z}(x) = \widetilde{\rho}(\varphi(x), z)$ belongs to the Newtonian-Sobolev space $N^{1,p}(X)$;

\noindent
$(2)$ there exists a function (\textit{an upper gradient of mapping $\varphi$}) $|D_{up} \varphi| \in L^p(X)$ that is an upper gradient of $[\varphi]_{z}$ for all $z \in \widetilde{X}$.

We define a norm functional (modular) in $V^{1,p}(X;\widetilde{X})$ in the following way:
$$
    \|\varphi\|_{V^{1,p}(X;\widetilde{X})} = \inf\limits_{|D_{up} \varphi|} \||D_{up} \varphi|\|_{L^p(X)}.
$$

It was proved in \cite{R04}, if $X$ is a separable complete metric space, then the class $V^{1,p}(X; \widetilde{X})$ coincides with the definition of the Sobolev-type classes given by Korevaar and Schoen in \cite{KS93}, and, by \cite{HKST}, for $1 < p < \infty$, it coincides with the Newtonian-Sobolev class $N^{1,p}(X; \widetilde{X})$ (see the remark below). Moreover (see \cite{R97, R04} and \cite[Theorem 6.3.20]{HKST}), the collection of all upper gradients of $\varphi \in V^{1,p}(X; \widetilde{X})$ is a closed convex lattice inside $L^p(X)$ and, if nonempty, contains a function $|D_{min}\varphi| \in L^p(X)$ of minimal $L^p$-norm and for all upper gradients $|D_{up}\varphi|$

$$
    |D_{min}\varphi(x)| \leq |D_{up}\varphi(x)| \quad \text{for a.e. } x \in X.
$$

This allows us to consider the equivalent norm functional on $V^{1,p}(X, \widetilde{X})$, $1<p<\infty$: 
$$
    \|\varphi\|_{V^{1,p}(X;\widetilde{X})} = \||D_{min} \varphi|\|_{L^p(X)}.
$$

It is shown in \cite[Theorem 5.1]{R97} that for the composition of a mapping $\varphi: X \to \widetilde{X}$ in $V^{1,p}(X; \widetilde{X})$ and a Lipschitz function $f: \widetilde{X} \to \mathbb{R}$ we have the analogue of the chain rule. Namely, the composition $f \circ \varphi$ belongs to $N^{1,p}(X)$ and the following inequality holds 
    $$
        \nabla_{min}(f \circ \varphi)(x) \leq |D_{min}\varphi(x)| \cdot \Lip(f(\varphi(x))), \quad \text{for a. e. } x \in X,
    $$
where $\Lip(f(\varphi(x)))$ is a Lipschitz constant
$$
\Lip(f(y))=\sup\limits_{z\in \widetilde{X}}\frac{|f(y)-f(z)|}{\widetilde{\rho}(y,z)},
$$
calculated at the point $y = \varphi(x)$.

The equivalent definition can be obtained if we define a Newtonian-Sobolev space $N^{1,p}(X; V)$, where $V$ is a Banach space and then extend this definition to the metric spaces $\widetilde{X}$ using the Kuratowski embedding $\kappa: \widetilde{X} \to l^\infty(\widetilde{X})$, $(k(y))(z) = \varphi_y(z) = d(y,z) - d(z,y_0)$, where $y_0$ is a fixed basepoint (see \cite{HKST}). In \cite{HKST} it was shown that all of these definitions of Sobolev classes are equivalent if $X$ and $\widetilde{X}=\varphi(X)$ are separable spaces.

\subsection{Homogeneous metric measure space}

 Recall that  \textit{a homogeneous space} $(X,\rho,\mu)$ is a doubling metric measure space with $\sigma$-finite measure that supports the following absorption property:
\vskip 0.1cm
    
\textit{There is a constant $c_1\geq 1$ such that for all $x_1,x_2\in X$ and $r>0$ if balls $B(x_1,r)$, $B(x_2,r)$ intersect, then $B(x_2,r)\subset B(x_1,c_1r)$.}
\vskip 0.1cm

It is well-known that the doubling metric measure spaces are separable and hence all the facts about Newtonian-Sobolev spaces $N^{1,p}(X)$ and Reshetnyak-Sobolev classes $V^{1,p}(X; \widetilde{X})$ are valid on homogeneous spaces $(X,\rho,\mu)$.

Let $\varphi: \Omega \to \widetilde{\Omega}$ be a homeomorphism of domains $\Omega,\widetilde\Omega\subset X$. We define the metric Jacobian of the homeomorphism $\varphi$ at a point $x\in\Omega$  by the rule
$$
J(x,\varphi) = \lim\limits_{r \to 0} \frac{\mu(\varphi(B(x,r) ))}{\mu(B(x,r))}.
$$
It is known \cite{VU04,VU05} that $J(x,\varphi)$ is finite a.e. in $\Omega$ and belongs to $L^1_{\loc}(\Omega)$.

Let us formulate the change of variable formula in the Lebesgue integral \cite{VU04,VU05} in the case of bi-measurable homeomorphisms.

\begin{thm}
 Let $\varphi: \Omega \to \widetilde{\Omega}$ be a bi-measurable homeomorphism of domains $\Omega,\widetilde{\Omega}\subset X$. 
Then a function $f: \widetilde{\Omega} \to [0, \infty]$ is measurable if and only if a function $(f\circ \varphi)\cdot J(\cdot,\varphi): \Omega  \to [0, \infty]$ is measurable. Moreover, the change of variable formula 
        \begin{equation}\label{cvf}
            \int_\Omega f\circ\varphi(x) J(x,\varphi) \, d\mu(x) = \int_{\widetilde{\Omega}} f(y) \, d\mu(y)
        \end{equation}
holds.
\end{thm}

The corollary of this theorem is the fact that $\mu(\varphi(Z)) = 0$ if $Z = \{x \in \Omega: J(x,\varphi) = 0\}$.

The structure of homogeneous spaces allows us effectively use the differentiability properties of set functions that we will associate with the norm of the composition operator. 

Recall that a nonnegative function $\Phi$ defined on open subsets of $\Omega\subset X$ is called a monotone countably additive set function \cite{VU04,VU05} if

\noindent
1) $\Phi(U_1)\leq \Phi(U_2)$ if $U_1\subset U_2\subset\Omega$;

\noindent
2)  for any collection $U_i \subset U \subset \Omega$, $i=1,2,...$, of mutually disjoint open sets
$$
\sum_{i=1}^{\infty}\Phi(U_i) = \Phi\left(\bigcup_{i=1}^{\infty}U_i\right).
$$

The following lemma gives properties of monotone countably additive set functions defined on open subsets of $\Omega\subset X$ \cite{VU04,VU05}.

\begin{lem}
\label{lem:AddFun}
Let $\Phi$ be a monotone countably additive set function defined on open subsets of the domain $\Omega\subset X$. Then

\noindent
(a) at almost all points $x\in \Omega$ there exists a finite derivative
$$
\lim\limits_{r\to 0}\frac{\Phi(B(x,r))}{|B(x,r)|}=\Phi'(x);
$$

\noindent
(b) $\Phi'(x)$ is a measurable function;

\noindent
(c) for every open set $U\subset \Omega$ the inequality
$$
\int\limits_U\Phi'(x)~d\mu(x)\leq \Phi(U)
$$
holds.
\end{lem}

The following version of the Lebesgue differentiation theorem on homogeneous spaces was given in  \cite[Corollary 3]{VU04}.
\begin{thm}
    Let $X=(X, \rho, \mu)$ be a homogeneous space, $U$ be a domain in $X$ and $f \in L^1_{\operatorname{loc}}(U)$, then, for almost all $x \in U$,
        $$
            \lim\limits_{r\to 0}\frac{1}{\mu(B(x, r))}\int_{B(x, r)} |f(y) - f(x)| \, d\mu(y) = 0.
        $$
\end{thm}

\section{Composition operators on Sobolev spaces}

Let $\Omega,\widetilde{\Omega} \subset X$ be bounded domains in the homogeneous space $X=(X, \rho, \mu)$.
The bi-measurable homeomorphism $\varphi: \Omega \to \widetilde{\Omega}$ generates a bounded composition operator
    $$
        \varphi^*: S^{1,p}(\widetilde\Omega) \to S^{1,q}(\Omega), \quad 1 < q \leq p < \infty,
    $$
by the composition rule $\varphi^*(f) = f \circ \varphi$, if $\varphi^*(f) \in S^{1,q}(\Omega)$, and there exists a constant $C_{p,q} < \infty$, such that 
    \begin{equation}\label{comp}
        \|\varphi^*(f)\|_{S^{1,q}(\Omega)} \leq C_{p,q} \|f\|_{S^{1,p}(\widetilde\Omega)}
    \end{equation}
for any function $f\in S^{1,p}(\widetilde\Omega)$.

\subsection{The capacitary characteristics of mappings generating composition operators}

In the following assertion we prove the capacitary Luzin $N^{-1}$ property for mappings generating bounded composition operators on Newtonian-Sobolev spaces.

\begin{thm}
\label{CapacityPQ}
Let a bi-measurable homeomorphism $\varphi :\Omega\to \widetilde{\Omega}$ of bounded domains $\Omega,\widetilde{\Omega} \subset X$
generate a bounded composition operator
$$
\varphi^{\ast}: S^{1,p}(\widetilde{\Omega})\to S^{1,q}(\Omega), \quad 1< q\leq p<\infty.
$$
Then the inequality
$$
\cp_{q}^{1/q}(\varphi^{-1}(\widetilde{E});\Omega)
\leq C_{p,q}\cp_{p}^{1/p}(\widetilde{E};\widetilde{\Omega})
$$
holds for every Borel set $\widetilde{E}\subset\widetilde{\Omega}$. 
\end{thm}

\begin{proof}

Let $F\subset E=\varphi^{-1}(\widetilde{E})$ be a compact set. Because $\varphi$ is a homeomorphism, $\widetilde{F}=\varphi(F)\subset\widetilde{E}$ is also a compact set.  Let $u\in C_0(\widetilde{\Omega})\cap N^{1,p}(\widetilde{\Omega})$ be an arbitrary admissible function such that $u\geq 1$ on $\widetilde{F}$. Then the composition $v=\varphi^{\ast}(u)$ belongs to $C_0(\Omega)\cap N^{1,q}({\Omega)}$, $v\geq 1$ on $F$ and 
$$
\|\varphi^{\ast}(u)\|_{S^{1,q}(\Omega)}\leq C_{p,q} \|u\|_{S^{1,p}(\widetilde{\Omega})}.
$$
Since the function $v=\varphi^{\ast}(u)\in C_0(\Omega)\cap N^{1,q}({\Omega)}$ is an admissible function for the compact $F\subset E$, we have the inequality
$$
\cp_{q}^{1/q}(\varphi^{-1}(\widetilde{F});\Omega)\leq \|\varphi^{\ast}(u)\|_{S^{1,q}(\Omega)}\leq C_{p,q} \|u\|_{S^{1,p}(\widetilde{\Omega})}.
$$
Taking the infimum over all functions $u\in C_0(\widetilde{\Omega})\cap N^{1,p}(\widetilde{\Omega})$ such that $u\geq 1$ on $\widetilde{F}$, we obtain 
$$
\cp_{q}^{1/q}(\varphi^{-1}(\widetilde{F});\Omega)
\leq C_{p,q}\cp_{p}^{1/p}(\widetilde{F};\widetilde{\Omega})
$$
for any compact set $\widetilde{F}\subset \widetilde{E}\subset \widetilde{\Omega}$.

Now for the Borel set $\widetilde{E}\subset \widetilde{\Omega}$ we have (by the definition of the $p$-capacity of Borel sets)
$$
\cp_{p}^{1/p}(\widetilde{F};\widetilde{\Omega})\leq \underline{\cp}_{p}^{1/p}(\widetilde{E};\widetilde{\Omega})=\cp_{p}^{1/p}(\widetilde{E};\widetilde{\Omega}).
$$
Hence 
$$
\cp_{q}^{1/q}(\varphi^{-1}(\widetilde{F});\Omega)\leq C_{p,q}\cp_{p}^{1/p}(\widetilde{E};\widetilde{\Omega}).
$$
Since $F=\varphi^{-1}(\widetilde{F})$ is an arbitrary compact set, $F\subset E$, $E$ is a Borel set as a preimage of the Borel set $\widetilde{E}$ under the homeomorphism $\varphi$, then
\begin{multline*}
\cp_{q}^{1/q}(\varphi^{-1}(\widetilde{E});\Omega)=\underline{\cp}_{q}^{1/q}(\varphi^{-1}(\widetilde{E});\Omega)=
\sup\limits_{F\subset E}\cp_{q}^{1/q}(F;\Omega)\\
\leq C_{p,q}\cp_{p}^{1/p}(\widetilde{E};\widetilde{\Omega}).
\end{multline*}

\end{proof}

\begin{cor}
\label{cor:Capacity}
Let a bi-measurable homeomorphism $\varphi :\Omega\to \widetilde{\Omega}$ of bounded domains $\Omega,\widetilde{\Omega} \subset X$
generate a bounded composition operator
$$
\varphi^{\ast}: S^{1,p}(\widetilde{\Omega})\to S^{1,q}(\Omega), \quad 1<q\leq p<\infty.
$$
Then the preimage of a set of $p$-capacity zero has $q$-capacity zero.
\end{cor}

\subsection{The analytical characteristics of mappings generating composition operators}

Since the homogeneous space $X=(X, \rho, \mu)$ is locally compact, we define \textit{local Sobolev spaces} $V^{1,p}_{\loc}(\Omega; \widetilde\Omega)$, $1 < p < \infty$, as follows: $\varphi \in V^{1,p}_{\loc}(\Omega; \widetilde\Omega)$ if and only if $\varphi \in V^{1,p}(U; \widetilde\Omega)$ for every open and bounded set $U \subset \Omega$ such that $\overline{U} \subset \Omega$, where $\overline{U}$ denotes the topological closure of $U$.

Let us define the inverse $q$-dilatation \cite{VU04,VU05} of a bi-measurable Sobolev homeomorphism $\varphi:\Omega\to\widetilde{\Omega}$ as
\begin{equation}
\label{p-dil}
H_q(y)=\left(\frac{|D_{min}\varphi(x)|^q}{J(x,\varphi)}\right)^{\frac{1}{q}}, \,\,x=\varphi^{-1}(y).
\end{equation}

Let us give the results regarding the boundedness of composition operators in Newtonian-Sobolev spaces. First, let us consider the case where $p=q$.

\begin{thm}\label{equal}
   Let $\Omega$, $\widetilde{\Omega}$ be bounded domains in the homogeneous space $X = (X, \rho, \mu)$. Then a bi-measurable homeomorphism $\varphi: \Omega \to \widetilde{\Omega}$ generates a bounded composition operator
    $$
        \varphi^*: S^{1,p}(\widetilde{\Omega}) \to S^{1,p}(\Omega), \quad \varphi^*(f) = f \circ \varphi, \quad 1 < p < \infty,
    $$
    if and only if $\varphi$ belongs to the Reshetnyak-Sobolev class $V^{1,p}_{\loc}(\Omega; \widetilde{\Omega})$ and 
		
		$$
		H_p(\varphi;\widetilde\Omega)=\ess\sup_{y\in \widetilde\Omega}H_p(y) < \infty.
		$$
\end{thm}

\begin{proof} \textit{Necessity.}
We prove that a bi-measurable homeomorphism $\varphi$ belongs to the Reshetnyak-Sobolev class $V^{1,p}_{\loc}(\Omega; \widetilde{\Omega})$.

Let $B\subset\Omega$ be an arbitrary ball, such that $\overline{B}\subset\Omega$. Then $\widetilde{B}=\varphi(B)\subset\widetilde{\Omega}$ be an open bounded set, such that $\overline{\widetilde{B}}\subset\widetilde{\Omega}$.
Suppose $z \in \widetilde{B}$ be an arbitrary basepoint. Then we consider test functions $\rho_z(y) = \rho(y, z)$, $z \in \widetilde{B}$. The functions $\rho_z$ are Lipschitz functions with 
$$
\nabla_{up} \rho_z (y)= \Lip \rho_z (y) = 1\,\, \text{for almost all}\,\, y\in \widetilde{B}
$$ 
and belong to the Newtonian-Sobolev space $S^{1,p}(\widetilde{B})$ for all $1 < p < \infty$. 

Since the composition operator $\varphi^*$ is bounded, the functions $[\varphi]_{z}(x)=\rho_z \circ \varphi (x) =  \rho(\varphi(x), z)$ belong to the Newtonian-Sobolev space $S^{1,p}(B)$ for all $z \in \widetilde{B}$. Hence, by the Sobolev embedding theorems \cite[Theorem 9.1.15]{HKST} the functions $[\varphi]_{z}(x)$ belong to the normed Newtonian-Sobolev space $N^{1,p}(B)$ for all $z \in \widetilde{B}$.

Because $B\subset\Omega$ is an arbitrary ball, such that $\overline{B}\subset\Omega$, then the functions $[\varphi]_{z}(x)$ belong to the normed Newtonian-Sobolev space $N^{1,p}(U)$ for all $z \in \widetilde{U}$, where $U\subset\Omega$ is an open bounded set, such that $\overline{U}\subset\Omega$ and $\widetilde{U}=\varphi(U)\subset\widetilde{\Omega}$.

Fix an arbitrary point $y_0 \in \widetilde{U}\subset\widetilde{\Omega}$. Now denote by $B(y_0, r)\subset\widetilde{U}$ a ball with a center at the point $y_0$ and the radius $r$. Since $\varphi$ is a homeomorphism, $\varphi^{-1}(B(y_0, r))$ is an open connected set and for every function $f \in S^{1,p}(\widetilde{U})$, such that $\supp f \subset B(y_0, r)$, the following inequality 
    \begin{equation}
		\label{eq1}
        \|\varphi^*(f)\|_{S^{1,p}(\varphi^{-1}(B))} \leq C_{p,p} \cdot \|f\|_{S^{1,p}(B)},\,\, B=B(y_0, r),
    \end{equation}
holds.
    
Let  $z \in \widetilde{U}$ be a basepoint. We consider test functions 
\begin{equation}
		\label{test_f}
f_z(y) = (\rho_z(y)-\rho_z(y_0)) \eta\left(\frac{\rho(y,y_0)}{r}\right),
\end{equation}
where $\eta: \mathbb{R} \to \mathbb{R}$,  $\eta \in C^\infty_0(\mathbb{R})$, is a smooth function equal to one on $B(0,1)$ and equal to zero outside a ball $B(0, 2)$. 

It follows that $(f_z \circ \varphi)(x) = [\varphi]_{z}$ for all $x \in \varphi^{-1}(B)$ and $\nabla_{min} (f_z \circ \varphi) = \nabla_{min} [\varphi]_{z}$ for almost all $x \in \varphi^{-1}(B)$. Hence we obtain
\begin{multline*}
\|\nabla_{min} [\varphi]_{z}\|_{L^p(\varphi^{-1}(B))}=\|\nabla_{min} f_z \circ \varphi\|_{L^p(\varphi^{-1}(B))}\\
 \leq \|\nabla_{min} f_z \circ \varphi\|_{L^p(\varphi^{-1}(2B))} = \|\varphi^*(f_z)\|_{S^{1,p}(\varphi^{-1}(2B))}. 
\end{multline*}

The test functions $f_z$ are the products of the 1-Lipschitz functions $\rho_z: X \to \mathbb{R}$ and the smooth function $\eta: \mathbb{R} \to \mathbb{R}$. Hence for such functions by the Leibniz rule (\cite[Theorem 2.15]{BB}) and the chain rule (\cite[Theorem 2.16]{BB}), we have
\begin{multline*}
\nabla_{min}f_z (y) \leq \nabla_{up} f_z(y) = \nabla_{min}(\rho_z(y)-\rho_z(y_0))\left|\eta\left(\frac{\rho(y,y_0)}{r}\right)\right| \\
+ |(\rho_z(y)-\rho_z(y_0))|\left|\eta'\left(\frac{\rho(y,y_0)}{r}\right)\right|\frac{\nabla_{min}\rho(y,y_0)}{r}\\
=
\nabla_{min}\rho_z(y)\left|\eta\left(\frac{\rho(y,y_0)}{r}\right)\right| 
+ \frac{|(\rho_z(y)-\rho_z(y_0))|}{r}\left|\eta'\left(\frac{\rho(y,y_0)}{r}\right)\right|.
\end{multline*}
Since $\rho_z$ are 1-Lipschitz functions, we have 
$$
|(\rho_z(y)-\rho_z(y_0))|\leq \rho(y,y_0)\leq r.
$$
Hence 
$$
\nabla_{min}f_z (y) \leq \left|\eta\left(\frac{\rho(y,y_0)}{r}\right)\right| 
+ \left|\eta'\left(\frac{\rho(y,y_0)}{r}\right)\right|
\leq 1+\left|\eta'\left(\frac{\rho(y,y_0)}{r}\right)\right|
$$
for almost all $y\in B(y_0,r)$.

Since $\eta \in C^\infty_0(\mathbb{R})$, then the derivative $\eta'\left(\frac{\rho(y,y_0)}{r}\right)$ is bounded and we obtain that there exists a constant $C<\infty$, such that 
$$
\nabla_{min}f_z (y) \leq 1+C,
$$
for almost all $y\in B(y_0,r)$.

Using the doubling condition for the measure $\mu$, we obtain the inequality
    $$
        \|f_z\|_{S^{1,p}(2B)} = \|\nabla_{min}f_z\|_{L^{p}(2B)} \leq (1 +C)(\mu(2B))^{\frac{1}{p}} \leq (1+C)C^{\frac{1}{p}}_\mu (\mu(B))^{\frac{1}{p}},\,\,\text{for any}\,\,z\in \widetilde{U},
    $$
with the constant $(1+C)C^{\frac{1}{p}}_\mu<\infty$.

Substituting the obtained estimates in the inequality (\ref{eq1}), we have
    $$
        \int_{\varphi^{-1}(B)} (\nabla_{min} [\varphi]_{z} (x))^p \, d\mu(x) \leq \widetilde{C}^p_{p,p}\mu(B),
    $$
    where $\widetilde{C}_{p,p} = (1+C)C^{\frac{1}{p}}_\mu C_{p,p}$
    
Hence, by the change of variables formula \eqref{cvf},
    $$
        \int\limits_{B} \frac{(\nabla_{min} [\varphi]_{z} (\varphi^{-1}(y)))^p}{J(\varphi^{-1}(y),\varphi)} \, d\mu(y) \leq \widetilde{C}^p_{p,p}\mu(B),
    $$
that implies
$$
       \frac{1}{\mu(B)} \int\limits_{B} \frac{(\nabla_{min} [\varphi]_{z} (\varphi^{-1}(y)))^p}{J(\varphi^{-1}(y),\varphi)} \, d\mu(y) \leq \widetilde{C}^p_{p,p}.
    $$

 By the Lebesgue Differentiation Theorem, it follows that
    \begin{equation}
		\label{eq2}
        \frac{(\nabla_{min} [\varphi]_{z} (\varphi^{-1}(y)))^p}{J(\varphi^{-1}(y),\varphi)} \leq \widetilde{C}^p_{p,p},
				\,\,\text{for almost all}\,\, y\in \widetilde{U}.
    \end{equation}
Since $\varphi$ is a bi-measurable homeomorphism, this inequality (\ref{eq2}) is equivalent to the inequality
\begin{equation}
\label{eq3}
		(\nabla_{min} [\varphi]_{z} (x))^p \leq \widetilde{C}^p_{p,p} J(x,\varphi)<\infty, \,\,z\in \widetilde{U},
\end{equation}
which holds for almost all  $x\in U$.

Hence, 
$$
\int\limits_{U}(\nabla_{min} [\varphi]_{z} (x))^p~d\mu(x) \leq \widetilde{C}^p_{p,p} \int\limits_{U} J(x,\varphi)~d\mu(x)<\infty,
$$
for all $z\in \widetilde{U}$ and for all open sets $U\subset\Omega$, such that $\overline{U}\subset\Omega$ and $\widetilde{U}=\varphi(U)\subset\widetilde{\Omega}$.
Therefore, $(J(x,\varphi))^{\frac{1}{p}}$ is an upper gradient of the homeomorphism $\varphi: U \to \widetilde{U}$, and $\varphi$ belongs to the Reshetnyak-Sobolev class $V^{1,p}_{\loc}(\Omega; \widetilde{\Omega})$. Moreover, as $1< p< \infty$, there exists a minimal upper gradient $|D_{min}\varphi| \in L^p_{\loc}(\Omega)$.

To finish the proof, since $\widetilde{U}\subset\widetilde{\Omega}$ is an arbitrary open set, we rewrite the inequality \eqref{eq2} in the following form:
\begin{equation*}
\frac{|D_{min} \varphi(\varphi^{-1}(y))|^p}{J(\varphi^{-1}(y),\varphi)} \leq \widetilde{C}^p_{p,p}
\quad \text{for almost all} \quad y \in \widetilde{\Omega}.
\end{equation*}

Then
$$
		H_p(\varphi;\widetilde\Omega)=\ess\sup_{y\in \widetilde\Omega}H_p(y)=\ess\sup_{y\in \widetilde\Omega}\left(\frac{|D_{min} \varphi(\varphi^{-1}(y))|^p}{J(\varphi^{-1}(y),\varphi)}\right)^{\frac{1}{p}}\leq \widetilde{C}_{p,p}< \infty.
$$

\noindent
\textit{Sufficiency.}  Let $H_p(\varphi; \widetilde{\Omega}) < \infty$. First we prove that the inequality  
$$
\|f \circ \varphi\|_{S^{1,p}(\Omega)} \leq C_l H_p(\varphi; \widetilde{\Omega}) \|f\|_{S^{1,p}(\widetilde{\Omega})},
$$
holds for any Lipschitz function $f \in S^{1,p}(\widetilde{\Omega}) \cap \Lip(\widetilde{\Omega})$, where a constant $C_l \geq 1$ depends on the domain $\Omega$.

In \cite{HKST} it was proved, that if $\varphi \in V^{1,p}(\Omega; \widetilde{\Omega})$, then $\varphi$ is absolutely continuous along $p$-a.e. rectifiable curves in $\Omega$. Since $f$ is a Lipschitz function, this implies (\cite{C,HKST}) that there exists a constant $C_l \geq 1$ that depends only on the domain $\Omega$, such that
$$
\nabla_{min}f (\varphi(x))  \leq \Lip(f (\varphi(x))) \leq C_l \nabla_{min} f(\varphi(x))
$$ 
for almost all $x\in\Omega$. Hence \cite{R97}
$$
\nabla_{min}(f \circ \varphi)(x)\leq |D_{min}\varphi(x)| \cdot \Lip(f (\varphi(x))) \leq C_l |D_{min}\varphi(x)| \cdot \nabla_{min} f(\varphi(x))
$$
for almost all $x\in\Omega$.

Then, applying the change of variables formula \eqref{cvf}, we obtain

\begin{multline*}
        \|f \circ \varphi\|_{S^{1,p}(\Omega)} = \left( \int\limits_{\Omega} \left(\nabla_{min}(f \circ \varphi)(x)\right)^p \, d\mu(x) \right)^{\frac{1}{p}}\\
        \leq C_l\left( \int\limits_{\Omega} |D_{min}\varphi(x)|^p \cdot (\nabla_{min} f(\varphi(x)))^p \, d\mu(x) \right)^{\frac{1}{p}}\\
				= C_l \left( \int\limits_{\Omega} \frac{|D_{min}\varphi(x)|^p}{J(x,\varphi)} \cdot (\nabla_{min} f (\varphi(x)))^p J(x,\varphi) \, d\mu(x) \right)^{\frac{1}{p}} \\
				= C_l \left( \int\limits_{\widetilde\Omega} \frac{|D_{min}\varphi(\varphi^{-1}(y))|^p}{J(\varphi^{-1}(y),\varphi)} \cdot (\nabla_{min} f(y))^p  \, d\mu(y) \right)^{\frac{1}{p}}.
\end{multline*}
Since
$$
		H_p(\varphi;\widetilde\Omega)=\ess\sup_{y\in \widetilde\Omega}H_p(y)  =\ess\sup_{y\in \widetilde\Omega} \left(\frac{|D_{min}\varphi(\varphi^{-1}(y))|^p}{J(\varphi^{-1}(y),\varphi)}\right)^{\frac{1}{p}}< \infty,
$$
then 
\begin{equation*}
        \|f \circ \varphi\|_{S^{1,p}(\Omega)} \leq C_l H_p(\varphi;\widetilde\Omega) \left( \int\limits_{\widetilde\Omega} (\nabla_{min} f(y))^p \, d\mu(y) \right)^{\frac{1}{p}}=C_l H_p(\varphi;\widetilde\Omega)\|f \|_{S^{1,p}(\widetilde\Omega)}
\end{equation*}
for any Lipschitz function $f \in S^{1,p}(\widetilde\Omega) \cap \Lip(\widetilde{\Omega})$.

To extend the estimate onto all functions $f\in S^{1,p}(\widetilde{\Omega})$, $1< p<\infty$, consider a sequence of Lipschitz functions $f_k\in S^{1,p}(\widetilde{\Omega})$, $k=1,2,...$, such that $f_k\to f$ in $S^{1,p}(\widetilde{\Omega})$ and $f_k\to f$ $p$-quasi-everywhere in $\widetilde{\Omega}$ as $k\to\infty$. Since by Corollary~\ref{cor:Capacity} the preimage $\varphi^{-1}(S)$ of the set $S\subset \widetilde{\Omega}$ of $p$-capacity zero has the $p$-capacity zero, we have $\varphi^{\ast}(f_k)\to \varphi^{\ast}(f)$
$p$-quasi-everywhere in $\Omega$ as $k\to\infty$. This observation leads us to the following conclusion: extension by continuity of the operator $\varphi^{\ast}$ from $S^{1,p}(\widetilde{\Omega})\cap \Lip(\widetilde{\Omega})$ to $S^{1,p}(\widetilde{\Omega})$ coincides with the composition operator $\varphi^{\ast}$, $\varphi^{\ast}(f) = f\circ\varphi$.
   
\end{proof}

Let us define the $p$-dilatation \cite{VU04,VU05} of a bi-measurable Sobolev homeomorphism $\varphi:\Omega\to\widetilde{\Omega}$
\begin{equation}
\label{dil}
K_p(x)=\left(\frac{|D_{min}\varphi(x)|^p}{J(x,\varphi)}\right)^{\frac{1}{p}}.
\end{equation}

Then we have:

\begin{thm}\label{equal_p}
    Let $\Omega$, $\widetilde{\Omega}$ be bounded domains in the homogeneous space $X = (X, \rho, \mu)$. Then the bi-measurable homeomorphism $\varphi: \Omega \to \widetilde{\Omega}$ generates a bounded composition operator
    $$
        \varphi^*: S^{1,p}(\widetilde{\Omega}) \to S^{1,p}(\Omega), \quad \varphi^*(f) = f \circ \varphi, \quad 1 < p < \infty,
    $$
    if and only if $\varphi$ belongs to the Reshetnyak-Sobolev class $V^{1,p}_{\loc}(\Omega; \widetilde{\Omega})$ and 
		
		$$
		K_p(\varphi;\Omega)=\ess\sup_{x\in \Omega}K_p(x) < \infty.
		$$
\end{thm}

\begin{proof}
Since $\varphi: \Omega \to \widetilde{\Omega}$ is a bi-measurable homeomorphism, then 
    $$
		K_p(\varphi;\Omega)=\ess\sup_{x\in \Omega}K_p(x) =H_p(\varphi;\widetilde\Omega)=\ess\sup_{y\in \widetilde\Omega}H_p(y).
	$$
\end{proof}

In the case $1<q<p<\infty$, we use the following countable-additive property of composition operators on Sobolev spaces.

\begin{lem}
\label{mainlem}
    Let $\Omega$, $\widetilde{\Omega}$ be bounded domains in the homogeneous space $X = (X, \rho, \mu)$. If the homeomorphism $\varphi: \Omega \to \widetilde{\Omega}$ generates a bounded composition operator
    $$
        \varphi^*: S^{1,p}(\widetilde\Omega) \to S^{1,q}(\Omega), \quad \varphi^*(f) = f \circ \varphi, \quad 1< q < p < \infty,
    $$
    then the function
    $$
        \Phi(\widetilde{A}) = \sup\limits_{f \in S^{1,p}(\widetilde{A}) \cap C_0(\widetilde{A})} \Bigg( \frac{\|\varphi^\ast(f)\|_{S^{1,q}(\Omega)}}{\|f\|_{S^{1,p}(\widetilde{A})}} \Bigg)^{\frac{pq}{p-q}},
    $$
    is a bounded monotone countable additive set function defined on open bounded subsets $\widetilde{A} \subset \widetilde\Omega$.
\end{lem}

The proof of this lemma repeats the proof of the corresponding lemmas from \cite{U93,VU04} and we don't give it in the present article.

The second main statement is the following theorem.

\begin{thm}\label{different}
    Let $\Omega$, $\widetilde{\Omega}$ be bounded domains in the homogeneous space $X = (X, \rho, \mu)$. The bi-measurable homeomorphism $\varphi: \Omega \to \widetilde{\Omega}$ generates a bounded composition operator
    $$
        \varphi^*: S^{1,p}(\widetilde\Omega) \to S^{1,q}(\Omega), \quad \varphi^*(f) = f \circ \varphi, \quad 1< q < p < \infty,
    $$
    if and only if $\varphi \in V^{1,q}_{\loc}(\Omega; \widetilde{\Omega})$  and 
    $$
		H_{p,q}(\varphi;\widetilde\Omega)=\left(\int\limits_{\widetilde\Omega} H_q^{\frac{pq}{p-q}}(y)~d\mu(y) \right)^{\frac{p-q}{pq}}< \infty.
	$$
\end{thm}

\begin{proof} \textit{Necessity.}
Let us first prove that a bi-measurable homeomorphism $\varphi$ belongs to the Reshetnyak-Sobolev class $V^{1,q}_{\loc}(\Omega; \widetilde{\Omega})$.

Let $B\subset\Omega$ be an arbitrary ball, such that $\overline{B}\subset\Omega$. Then $\widetilde{B}=\varphi(B)\subset\widetilde{\Omega}$ be an open bounded set, such that $\overline{\widetilde{B}}\subset\widetilde{\Omega}$.
Suppose $z \in \widetilde{B}$ be an arbitrary basepoint. Then we consider test functions $\rho_z(y) = \rho(y, z)$, $z \in \widetilde{B}$. The functions $\rho_z$ are Lipschitz functions with 
$$
\nabla_{up} \rho_z (y)= \Lip \rho_z (y) = 1\,\, \text{for almost all}\,\, y\in \widetilde{B}
$$ 
and belong to the Newtonian-Sobolev space $S^{1,p}(\widetilde{B})$ for all $1 < p < \infty$. 

Since the composition operator $\varphi^*$ is bounded, the functions $[\varphi]_{z}(x)=\rho_z \circ \varphi (x) =  \rho(\varphi(x), z)$ belong to the Newtonian-Sobolev space $S^{1,q}(B)$ for all $z \in \widetilde{B}$. Hence, by the Sobolev embedding theorems \cite[Theorem 9.1.15]{HKST} the functions $[\varphi]_{z}(x)$ belong to the normed Newtonian-Sobolev space $N^{1,q}(B)$ for all $z \in \widetilde{B}$.

Because $B\subset\Omega$ is an arbitrary ball, such that $\overline{B}\subset\Omega$, then the functions $[\varphi]_{z}(x)$ belong to the normed Newtonian-Sobolev space $N^{1,q}(U)$ for all $z \in \widetilde{U}$, where $U\subset\Omega$ is an open bounded set, such that $\overline{U}\subset\Omega$ and $\widetilde{U}=\varphi(U)\subset\widetilde{\Omega}$.

By Lemma \ref{mainlem} for every open set $\widetilde{A} \subset \widetilde{U}$ and for every continuous function $f \in S^{1,p}(\widetilde{U})$, $\supp f \subset \widetilde{A}$, the following inequality holds:
    \begin{equation*}
		%\label{boundonsets}
        \|\varphi^*(f)\|_{S^{1,q}(\varphi^{-1}(\widetilde{A}))} \leq (\Phi(\widetilde{A}))^{\frac{p-q}{pq}} \|f\|_{S^{1,p}(\widetilde{A})}.
    \end{equation*}

Fix an arbitrary point $y_0 \in \widetilde{U}\subset\widetilde{\Omega}$. Now denote by $B(y_0, r)\subset\widetilde{U}$ a ball with a center at the point $y_0$ and a radius $r$. Since $\varphi$ is a homeomorphism, $\varphi^{-1}(B(y_0, r))$ is an open connected set.    
Let  $z \in \widetilde{U}$ be a basepoint. We consider test functions defined by the equation (\ref{test_f})
$$
f_z(y) = (\rho_z(y) -\rho_z(y_0)) \eta\left(\frac{\rho(y,y_0)}{r}\right).
$$
By using the estimates of the test functions $f_z$ in Sobolev spaces, obtained in the proof of Theorem~\ref{equal}, we have
    $$
        \int_{\varphi^{-1}(B)} (\nabla_{min} [\varphi]_{z} (x))^q \, d\mu(x) \leq \widetilde{C}^q_{p,q} (\Phi(2B))^{\frac{p-q}{p}} (\mu(B))^{\frac{q}{p}}.
    $$

    Hence, by the the change of variables formula \eqref{cvf},
    $$
        \int_{B} \frac{(\nabla_{min} [\varphi]_{z} (\varphi^{-1}(y)))^q}{J(\varphi^{-1}(y),\varphi)} \, d\mu(y) \leq \widetilde{C}^q_{p,q} (\Phi(2B))^{\frac{p-q}{p}}  (\mu(B))^{\frac{q}{p}}.
    $$

    Further, divide both sides on $\mu(B)$ and use the Lebesgue Differentiation Theorem with $r \to 0$:
    \begin{equation}\label{pqdilatatation}
        \Big( \frac{(\nabla_{min} [\varphi]_{z} (\varphi^{-1}(y)))^q}{J(\varphi^{-1}(y),\varphi)} \Big)^{\frac{p}{p-q}} \leq \widetilde{C}^{\frac{pq}{p-q}}_{p,q} C_\mu \Phi'(y) \quad \text{for almost all } y \in \widetilde{U}.
    \end{equation}

    Since $\varphi$ is a bi-measurable homeomorphism, the above inequality equivalents to the equation
\begin{equation}
\label{eqpoint}
		(\nabla_{min} [\varphi]_{z} (x))^q \leq \widetilde{C}^q_{p,q} C_\mu (\Phi'(\varphi(x)))^{\frac{p-q}{p}} J(x,\varphi)<\infty,
\end{equation}
which holds for almost all  $x\in U$ and for all $z\in \widetilde{U}$.

Hence, 
\begin{multline*}
    \int\limits_{U}(\nabla_{min} [\varphi]_{z} (x))^q~d\mu(x) \leq  \widetilde{C}^q_{p,q} C_\mu \int\limits_{U} (\Phi'(\varphi(x))^{\frac{p-q}{p}} J(x,\varphi)~d\mu(x) \\
    = \widetilde{C}^q_{p,q} C_\mu \int\limits_{\widetilde{U}} (\Phi'(y))^{\frac{p-q}{p}}~d\mu(y) <\infty.
\end{multline*}
Therefore, $(\Phi'(\varphi(x))^{\frac{p-q}{p}} J(x,\varphi)$ is an upper gradient of the homeomorphism $\varphi: U \to \widetilde{U}$ and $\varphi$ belongs to the Reshetnyak-Sobolev class $V^{1,q}_{\loc}(\Omega; \widetilde{\Omega})$. Moreover, since $1< q < \infty$, there exists a minimal upper gradient $|D_{min}\varphi| \in L^q_{\loc}(\Omega)$.

    To finish the proof, since $\widetilde{U}\subset\widetilde{\Omega}$ is an arbitrary open set, we infer from the inequality \eqref{pqdilatatation} the following relation:
    $$
         \Big( \frac{|D_{min}\varphi (\varphi^{-1}(y))|^q}{J(\varphi^{-1}(y),\varphi)} \Big)^{\frac{p}{p-q}} \leq \widetilde{C}^{\frac{pq}{p-q}}_{p,q} C_\mu \Phi'(y) \quad \text{for almost all } y \in \widetilde{\Omega}.
    $$
    
    Then, integrating over $\widetilde\Omega$, we obtain
    \begin{multline*}
        H^{\frac{pq}{p-q}}_{p,q}(\varphi; \widetilde\Omega) = \int_{\widetilde\Omega} \Big(\frac{|D_{min} \varphi (\varphi^{-1}(y))|^q}{J(\varphi^{-1}(y),\varphi)} \Big)^{\frac{p}{p-q}} \, d\mu(y) \\
        \leq \widetilde{C}^q_{p,q} C_\mu \int_{\widetilde\Omega} \Phi'(y) \, d\mu(y) \leq \widetilde{C}^q_{p,q} C_\mu \Phi(\widetilde\Omega) \leq \widetilde{C}^q_{p,q} C_\mu \|\varphi^*\|^{\frac{pq}{p-q}}.
    \end{multline*}

    \textit{Sufficiency.}
    Assume that $H_{p,q}(\varphi; \widetilde\Omega) < \infty$.
    Let us proof that for an arbitrary Lipschitz function $f \in S^{1,p}(\widetilde\Omega) \cap \Lip(\widetilde\Omega)$ the following inequality holds: 
    $$
        \|f \circ \varphi\|_{S^{1,q}(\Omega)} \leq C_l H_{p,q}(\varphi; \widetilde\Omega) \|f\|_{S^{1,p}(\widetilde\Omega)},
    $$
	where a constant $C_l \geq 1$ depends on the domain $\Omega$.

    In \cite{HKST} it was proved that if $\varphi \in V^{1,q}(\Omega; \widetilde{\Omega})$, then $\varphi$ is absolutely continuous along $q$-a.e. rectifiable curves in $\Omega$. Together with the fact that $f$ is Lipschitz, this implies (\cite{C, HKST}) that $\Lip(f \circ \varphi)$ is an upper gradient of $f \circ \varphi$ and there exists a constant $C_l \geq 1$, such that
    $$
        \nabla_{min}f(\varphi(x))  \leq \Lip(f(\varphi(x))) \leq C_l  \nabla_{min}f(\varphi(x))
    $$ 
    for almost all $x\in\Omega$. Hence \cite{R97}
    $$
        \nabla_{min}(f \circ \varphi)(x)\leq |D_{min}\varphi(x)| \cdot \Lip(f (\varphi(x))) \leq C_l |D_{min}\varphi(x)| \cdot \nabla_{min} f(\varphi(x))
    $$
    for almost all $x\in\Omega$.

    Then, applying the change of variables formula \eqref{cvf} and H\"older inequality, we obtain
    \begin{multline*}
        \|f \circ \varphi\|^q_{S^{1,q}(\Omega)}  = \int_{\Omega} (\nabla_{min}(f \circ \varphi)(x))^q \, d\mu(x) \\
        \leq C_l \int_{\Omega} |D_{min}\varphi(x)|^q \cdot (\nabla_{min} f(\varphi(x)))^q \, d\mu(x) \\
        \leq C_l \int_{\widetilde\Omega} (\nabla_{min} f(y))^q \frac{|D_{min} \varphi(\varphi^{-1}(y))|^q}{J(\varphi^{-1}(y), \varphi)} \, d\mu(y) \\
        \leq C_l \Bigg( \int_{\widetilde\Omega} (\nabla_{min} f(y))^p \, d\mu(y)\Bigg)^{\frac{q}{p}} \Bigg( \int_{\widetilde\Omega} \Big(\frac{|D_{min} \varphi(\varphi^{-1}(y))|^q}{J(\varphi^{-1}(y), \varphi)} \Big)^{\frac{p}{p-q}} \, d\mu(y) \Bigg)^{\frac{p-q}{p}}.
    \end{multline*}

    Note that in the above inequalities we used the property $\mu(\varphi(Z)) = 0$ for $Z = \{x \in \Omega: J(x,\varphi) = 0\}$. 

    Since $H_{p,q}(\varphi; \widetilde\Omega) < \infty$, for any Lipschitz function we have
    $$
        \|f \circ \varphi\|_{S^{1,q}(\Omega)} \leq C_l H_{p,q}(\varphi; \widetilde\Omega) \|f\|_{S^{1,p}(\widetilde\Omega)}.
    $$

    To extend the estimate onto all functions $f\in S^{1,p}(\widetilde{\Omega})$, $1< p<\infty$, consider a sequence of Lipschitz functions $f_k\in S^{1,p}(\widetilde{\Omega})$, $k=1,2,...$, such that $f_k\to f$ in $S^{1,p}(\widetilde{\Omega})$ and $f_k\to f$ $p$-quasi-everywhere in $\widetilde{\Omega}$ as $k\to\infty$. Since by Corollary~\ref{cor:Capacity} the preimage $\varphi^{-1}(S)$ of the set $S\subset \widetilde{\Omega}$ of $p$-capacity zero has the $q$-capacity zero, we have $\varphi^{\ast}(f_k)\to \varphi^{\ast}(f)$ $q$-quasi-everywhere in $\Omega$ as $k\to\infty$. This observation leads us to the following conclusion: extension by continuity of the operator $\varphi^{\ast}$ from $S^{1,p}(\widetilde{\Omega})\cap \Lip(\widetilde{\Omega})$ to $S^{1,p}(\widetilde{\Omega})$ coincides with the composition operator $\varphi^{\ast}$, $\varphi^{\ast}(f) = f\circ\varphi$.
\end{proof}

As in the case $p=q$, we can formulate the above theorem in the terms of the $p$-dilatation function:

\begin{thm}\label{different_pq}
    Let $\Omega$, $\widetilde{\Omega}$ be bounded domains in the homogeneous space $X = (X, \rho, \mu)$. Then a bi-measurable homeomorphism $\varphi: \Omega \to \widetilde{\Omega}$ generates a bounded composition operator
    $$
        \varphi^*: S^{1,p}(\widetilde{\Omega}) \to S^{1,q}(\Omega), \quad \varphi^*(f) = f \circ \varphi, \quad 1 < q < p < \infty,
    $$
    if and only if $\varphi$ belongs to the Reshetnyak-Sobolev class $V^{1,q}_{\loc}(\Omega; \widetilde{\Omega})$ and 	
	$$
		K_{p,q}(\varphi;\Omega)= \left(\int\limits_{\Omega} K_q^{\frac{pq}{p-q}}(x)~d\mu(x) \right)^{\frac{p-q}{pq}}< \infty.
	$$
\end{thm}

\section{Sobolev type embedding theorems}

In accordance with \cite{VU04,VU05}, bi-measurable homeomorphisms $\varphi: \Omega \to \widetilde{\Omega}$, of bounded domains $\Omega,\widetilde{\Omega}\subset X$, belonging to the Reshetnyak-Sobolev class $V^{1,q}_{\loc}(\Omega; \widetilde{\Omega})$, such that
$$
K_{p,q}(\varphi;\Omega)= \left(\int\limits_{\Omega} K_q^{\frac{pq}{p-q}}(x)~d\mu(x) \right)^{\frac{p-q}{pq}} < \infty,\,\,1<q<p<\infty,
$$
and 
$$
K_p(\varphi;\Omega)=\ess\sup_{x\in \Omega}K_p(x) < \infty,\,\,1<q=p<\infty,
$$
are called \textit{weak $(p,q)$-quasiconformal mappings}. In general, weak $(p,q)$-quasi\-con\-for\-mal mappings don't generate bounded composition operators on normed Newtonian-Sobolev spaces:
$$
\varphi^\ast: N^{1,p}(\widetilde{\Omega}) \to N^{1,q}(\Omega),\,\, 1<q\leq p<\infty.
$$

Below, we give sufficient conditions on weak $(p,q)$-quasiconformal mappings that ensure the boundedness of composition operators in the case of normed spaces. In the case of Euclidean domains this result was proved in \cite{GGu,GU}.

\subsection{Composition operators on normed Sobolev spaces}

Recall that a bounded domain $\Omega \subset X$ supports the $(r,q)$-Poincar\'e-Sobolev inequality, $1 < q, r < \infty$, if there exists a constant $B_{q,r}(\Omega) < \infty$, such that the inequality
$$
\inf\limits_{c \in \mathbb{R}} \|g - c\|_{L^r(\Omega)} \leq B_{q,r}(\Omega) \|\nabla_{\min} g\|_{L^q(\Omega)}
$$
holds for all functions $g \in S^{1,q}(\Omega)$.

We also need a theorem on boundedness of the composition operator in the case of Lebesgue spaces. This theorem was proved for homogeneous type spaces in \cite{VU04}.

\begin{thm}\cite[Theorem 4]{VU04}\label{L_p}
    Let $\Omega$, $\widetilde{\Omega}$ be bounded domains in the homogeneous space $X = (X, \rho, \mu)$ and let $\psi: \widetilde\Omega \to \Omega$ be a homeomorphism. Then $\psi$ generates a bounded composition operator
    $$
        \psi^\ast: L^r(\Omega) \to L^s(\widetilde{\Omega}), \quad 1 \leq s \leq r < \infty,
    $$
    if and only if it is measurable and
    $$
        I_{r,s}(\psi^{-1}; \Omega) = \left( \int_{\Omega} J^{\frac{r}{r-s}}(x, \psi^{-1}) \, d\mu(x)\right)^{\frac{r-s}{r}} < \infty.
    $$
    In the case $r=s$ we assume that the metric Jacobian is in $L^\infty(\Omega)$.
\end{thm}

The results on composition operators on the normed Newtonian-Sobolev spaces states as follows:

\begin{thm}
\label{normedcomp}
Let $\Omega$, $\widetilde{\Omega}$ be bounded domains in the homogeneous space $X = (X, \rho, \mu)$. Suppose that $\Omega$ supports the $(r,q)$-Poincar\'e-Sobolev inequality, $1 < q\leq r < \infty$. Then a weak $(p,q)$-quasiconformal mapping $\varphi: \Omega \to \widetilde{\Omega}$ generates a bounded composition operator
    $$
        \varphi^*: N^{1,p}(\widetilde{\Omega}) \to N^{1,q}(\Omega), \quad \varphi^*(f) = f \circ \varphi, \quad 1 < q \leq p <\infty,
    $$
    if 
    $$
        I_{r,s}(\varphi; \Omega) = \left( \int_{\Omega} \left(J(x, \varphi)\right)^{\frac{r}{r-s}} \, d\mu(x)\right)^{\frac{r-s}{r}} < \infty.
    $$
    for some $s \in [p,r]$.
\end{thm}

\begin{proof}
Let a function $f \in N^{1,p}(\widetilde{\Omega})$. The assumptions of the theorem imply that the composition operator
$$
        \varphi^*: S^{1,p}(\widetilde{\Omega}) \to S^{1,q}(\Omega), \quad \varphi^*(f) = f \circ \varphi, \quad 1 < q \leq p < \infty,
$$
is bounded. Hence, a function $g = \varphi^{\ast}(f)$ belongs to the space $S^{1,q}(\Omega)$ and the following inequality holds:
\begin{equation}\label{compL1pq}
\|\varphi^{\ast}(f)\|_{S^{1,q}(\Omega)} \leq C_{p,q} \cdot \|f\|_{S^{1,p}(\widetilde{\Omega})}\leq C_{p,q} \cdot \|f\|_{N^{1,p}(\widetilde{\Omega})}.
\end{equation}
To prove the theorem, we need to estimate the norm of $g = f \circ \varphi$ in the Lebesgue space $L^q(\Omega)$ via the norm of $f$ in the Newtonian-Sobolev space $N^{1,p}(\widetilde{\Omega})$.

Because the domain $\Omega$ satisfies the $(r,q)$-Poincar\'e-Sobolev inequality and a function $g \in S^{1,q}(\Omega)$, we obtain that a function $g$ is in $L^r(\Omega)$. By using the triangle inequality and the H\"older inequality in the case $r > q$, we obtain the following estimate 
    $$
        \|g\|_{L^q(\Omega)} \leq \|c\|_{L^q(\Omega)} + \|g - c\|_{L^q(\Omega)} \leq |c|(\mu(\Omega))^{\frac{1}{q}} + (\mu(\Omega))^{\frac{r-q}{r}}\|g-c\|_{L^r(\Omega)},
    $$
which	holds for an arbitrary constant $c \in \mathbb{R}$.
		
By the similar way we can estimate the module of a constant $c$:
    $$
        |c| = (\mu(\widetilde{\Omega}))^{-\frac{1}{p}}\|c\|_{L^p(\widetilde{\Omega})} \leq (\mu(\widetilde{\Omega}))^{-\frac{1}{p}}(\|f\|_{L^p(\widetilde{\Omega})} + \|f-c\|_{L^p(\widetilde{\Omega})})
    $$

Consider the term $\|f-c\|_{L^p(\widetilde{\Omega})}$. By the assumptions of the theorem, the composition operator
    $$
        (\varphi^{-1})^\ast: L^r(\Omega) \to L^s(\widetilde{\Omega}), \quad 1 < s \leq r < \infty
    $$
    is bounded for all $g \in L^r(\Omega)$ and so we have the inequality
    \begin{equation}\label{compLpq}
        \|g\circ \varphi^{-1}\|_{L^s(\widetilde{\Omega})} \leq I_{r,s}(\varphi; \Omega) \|g\|_{L^r(\Omega)}.
    \end{equation}
    
Because of the boundedness of the compisition operator $\varphi^\ast: S^{1,p}(\widetilde{\Omega}) \to S^{1,q}(\Omega)$ and the fact that the domain $\Omega$ satisfies the $(r,q)$-Poincar\'e-Sobolev inequality, this implies that
    \begin{multline*}
        \inf_{c \in \mathbb{R}}\|f-c\|_{L^s(\widetilde{\Omega})} \leq I_{r,s}(\varphi; \Omega) \inf_{c \in \mathbb{R}}\|g-c\|_{L^r(\widetilde{\Omega})} \\
        \leq B_{q,r}(\Omega) \cdot I_{r,s}(\varphi; \Omega)\|\nabla_{min}g\|_{L^q(\Omega)} \\
        \leq C_{p,q} \cdot B_{q,r}(\Omega) \cdot I_{r,s}(\varphi; \Omega) \|f\|_{S^{1,p}(\widetilde{\Omega})}.
    \end{multline*}

    In particular, this inequality implies that $f \in L^s(\widetilde{\Omega})$ and we can estimate the norm $\|f-c\|_{L^p(\widetilde{\Omega})}$ with the H\"older inequality if $s \geq p$:
    $$
        \|f-c\|_{L^p(\widetilde{\Omega})} \leq (\mu(\widetilde{\Omega}))^{\frac{s-p}{p}}\|f-c\|_{L^s(\widetilde{\Omega})}.
    $$

    Using all the above estimates, we obtain
    \begin{multline*}
        \|g\|_{L^q(\Omega)} \leq |c|(\mu(\Omega))^{\frac{1}{q}} + (\mu(\Omega))^{\frac{r-q}{r}}\|g-c\|_{L^r(\Omega)} \\
        \leq (\mu(\Omega))^{\frac{1}{q}}(\mu(\widetilde{\Omega}))^{-\frac{1}{p}}(\|f\|_{L^p(\widetilde{\Omega})} + (\mu(\widetilde{\Omega}))^{\frac{s-p}{p}}\|f-c\|_{L^s(\widetilde{\Omega})}) + (\mu(\Omega))^{\frac{r-q}{r}}\|g-c\|_{L^r(\Omega)}.
    \end{multline*}

    Taking the infimum over all $c \in \mathbb{R}$, we finally get
    \begin{multline*}
        \|g\|_{L^q(\Omega)} \leq (\mu(\Omega))^{\frac{1}{q}}(\mu(\widetilde{\Omega}))^{-\frac{1}{p}}\|f\|_{L^p(\widetilde{\Omega})} \\
        + (\mu(\Omega))^{\frac{1}{q}}(\mu(\widetilde{\Omega}))^{\frac{s-p-1}{p}}C_{p,q} \cdot B_{q,r}(\Omega) \cdot I_{r,s}(\varphi; \Omega) \|f\|_{S^{1,p}(\widetilde{\Omega})}  \\
        + (\mu(\Omega))^{\frac{r-q}{r}}C_{p,q} \cdot B_{q,r}(\Omega) \cdot I_{r,s}(\varphi; \Omega) \|f\|_{S^{1,p}(\widetilde{\Omega})} 
        \leq C\|f\|_{N^{1,p}(\widetilde{\Omega})}
    \end{multline*}
    and the result follows.
\end{proof}

\subsection{Relative embedding theorem}

The consequence of the above theorem is the following embedding theorem for Newtonian-Sobolev spaces defined on a wide class of domains of the homogeneous metric space $(X, \rho, \mu)$.

\begin{thm}
\label{embedding_1}
 Let $\Omega$, $\widetilde{\Omega}$ be bounded domains in the homogeneous space $X = (X, \rho, \mu)$. Suppose  that the embedding operator
 $$
    i: N^{1,q}(\Omega) \to L^{r}(\Omega), \,\,1<q \leq r<\infty,
 $$
 is bounded (compact). If there exists a  weak $(p,q)$-quasiconformal mapping $\varphi: \Omega \to \widetilde{\Omega}$, $1<q\leq p<\infty$, with
    $$
        I_{r,s}(\varphi; \Omega) = \left( \int_{\Omega} J^{\frac{r}{r-s}}(x, \varphi) \, d\mu(x)\right)^{\frac{r-s}{r}} < \infty.
    $$
    for $s \in [p,r]$, then the embedding operator
 $$
    \widetilde{i}: N^{1,p}(\widetilde{\Omega}) \to L^{s}(\widetilde\Omega)
 $$
 is bounded (respectively, compact).
\end{thm}

\begin{proof}
    Since the embedding operator $i: N^{1,q}(\Omega) \to L^{r}(\Omega)$ is bounded, then the domain $\Omega$ supports the $(r,q)$-Poincar\'e Sobolev inequality and the conditions of Theorem \ref{normedcomp} are fulfilled. Hence, the composition operator 
		$$
        \varphi^*: N^{1,p}(\widetilde{\Omega}) \to N^{1,q}(\Omega), \quad \varphi^*(f) = f \circ \varphi, \quad 1 < q \leq p <\infty,
    $$
is bounded. Because $I_{r,s}(\varphi; \Omega)<\infty$, the composition operator
$$
        \left(\varphi^{-1}\right)^\ast: L^r(\Omega) \to L^s(\widetilde{\Omega}), \quad 1 < s \leq r < \infty,
$$
is also bounded.

Therefore the embedding operator $\widetilde{i}$ is bounded (compact) as a superposition of bounded composition operators and the bounded (compact) embedding operator $i$. The norm of the operator $\widetilde{i}$ can be estimated as
    $$
        \|\widetilde{i}\| \leq \|\varphi^\ast\|\cdot \|i\| \cdot\|(\varphi^{-1})^\ast\| = C_{p,q} \cdot \|i\| \cdot I_{r,s}(\varphi, \Omega).
    $$
\end{proof}

\subsection{Embedding theorem in weak $(p,q)$-quasiconformal $\alpha$-regular domains}

Taking the previous theorem into account, we can formulate the more explicit result for the special class of domains. Let $\widetilde{\Omega}\subset X$ be a bounded domain. Then $\widetilde{\Omega}$ is called a \textit{weak $(p,q)$-quasiconformal $\alpha$-regular domain, $1<q\leq p<\infty$, $\alpha > 1$,} if there exists a weak $(p,q)$-quasiconformal mapping $\varphi: \widetilde{B} \to \widetilde{\Omega}$, where $\widetilde{B}\subset X$ is a ball with a radius one, such that
$$
I_\alpha(\varphi; \widetilde{B}) = \int_{\widetilde{B}} |J(x, \varphi)|^\alpha \, d\mu(x) < \infty,\,\, \alpha<\infty,
$$
and 
$$
I_\infty(\varphi; \widetilde{B}) = \ess\sup_{x\in \widetilde{B}} |J(x, \varphi)|< \infty,\,\,\text{if}\,\, \alpha=\infty.
$$

In the case of the Euclidean domains the notion of conformal $\alpha$-regular domains for the Sobolev embedding theorems was introduced in \cite{BGU16}.

Recall that a metric measure space $X = (X, \rho, \mu)$ satisfies a relative lower volume decay of order $\nu > 0$ with 
$$
    \nu = \log_2 C_\mu.
$$

\begin{thm}
\label{embedding_2}
 Let $\widetilde{\Omega}\subset X = (X, \rho, \mu)$ be a weak $(p,q)$-quasiconformal $\alpha$-regular domain, $1<q<\nu$. Then the embedding operator
 $$
    \widetilde{i}: N^{1,p}(\widetilde{\Omega}) \to L^{s}(\widetilde\Omega)
 $$
 is bounded for $s\leq r\frac{\alpha-1}{\alpha}$, $r=\frac{\nu q}{\nu-q}$, and is compact for $s< r\frac{\alpha-1}{\alpha}$.
\end{thm}

\begin{proof}

It has been proved \cite[Theorem 9.1.15]{HKST} that the embedding operator 
$$
    i: N^{1,q}(\widetilde{B}) \to L^{r}(\widetilde{B}), \,\,r=\frac{\nu q}{\nu-q},
$$  
is bounded. Hence by Theorem~\ref{embedding_1} the embedding operator
$$
    \widetilde{i}: N^{1,p}(\widetilde{\Omega}) \to L^{s}(\widetilde\Omega)
 $$
 is bounded for $s\leq r\frac{\alpha-1}{\alpha}$, $r=\frac{\nu q}{\nu-q}$, and is compact for $s< r\frac{\alpha-1}{\alpha}$.

\end{proof}

Let $\Omega$, $\widetilde{\Omega}$ be domains in the homogeneous space $X = (X, \rho, \mu)$. Then the homeomorphism $\varphi: \Omega \to \widetilde{\Omega}$ is called bi-Lipschitz, if there exists a constant $K_{\varphi}<\infty$, such that
$$
\frac{1}{K_{\varphi}}\leq\frac{\rho(\varphi(x_1),\varphi(x_2))}{\rho(x_1,x_2)}\leq K_{\varphi}, \,\, \text{for all}\,\, x_1,x_2\in\Omega.
$$
Let a domain $\widetilde{\Omega}\subset X = (X, \rho, \mu)$ be such, that there exists a bi-Lipschitz homeomorphism $\varphi: \widetilde{B} \to \widetilde{\Omega}$, where $\widetilde{B}$ is the unit ball in $X$, then $\widetilde{\Omega}$ is called a Lipschitz domain.

Because a Lipschitz domain $\widetilde{\Omega}\subset X$ is a  weak $(p,p)$-quasiconformal $\infty$-regular domain, $1<p<\infty$, then as a consequence of Theorem~\ref{embedding_2} we obtain.

\begin{thm}
\label{embedding_3}
 Let $\widetilde{\Omega}\subset X = (X, \rho, \mu)$ be a Lipschitz domain, Then for $1<p<\nu$, the embedding operator
 $$
    \widetilde{i}: N^{1,p}(\widetilde{\Omega}) \to L^{s}(\widetilde\Omega)
 $$
 is bounded for $s\leq r$, $r=\frac{\nu p}{\nu-p}$, and is compact for $s< r$.
\end{thm}

\vskip 0.1cm

Alexander Menovschikov; Department of Mathematics, HSE University, Moscow, Russia
 
\emph{E-mail address:} \email{menovschikovmath@gmail.com} \\

Alexander Ukhlov; Department of Mathematics, Ben-Gurion University of the Negev, P.O.Box 653, Beer Sheva, 8410501, Israel 
							
\emph{E-mail address:} \email{ukhlov@math.bgu.ac.il}

\end{document}